\documentclass[12pt]{article}
\usepackage{amssymb,amsmath,mathrsfs,amsthm,latexsym}
\usepackage{hyperref}
\usepackage{graphicx}
\usepackage{color}
\usepackage[OT2,OT1]{fontenc}
\usepackage{upquote}
\usepackage[numbers,sort&compress]{natbib}
\usepackage{alphabeta}
\usepackage{alltt}

\newtheorem{theorem}{T{\hskip 0pt\footnotesize\bf HEOREM}}[section]
\newtheorem{lemma}[theorem]{L{\hskip 0pt\footnotesize\bf EMMA}}

\begin{document}

\title{\Large\bf A rational approximation of the two-term Machin-like formula for $\pi$}

%\bigskip
\author{
\normalsize\bf Sanjar M. Abrarov, Rehan Siddiqui, Rajinder Kumar Jagpal \\
\normalsize\bf and Brendan M. Quine}

\date{July 23, 2024}
\maketitle
%\vspace{1cm}%\bigskip

\begin{abstract}
In this work, we consider the properties of the two-term Machin-like formula and develop an algorithm for computing digits of $\pi$ by using its rational approximation. In this approximation, both terms are constructed by using a representation of $1/\pi$ in the binary form. This approach provides the squared convergence in computing digits of $\pi$ without any trigonometric functions and surd numbers. The Mathematica codes showing some examples are presented.
\vspace{0.2cm}
\\
\noindent {\bf Keywords:} constant pi; iteration; nested radicals; rational approximation
\\
\end{abstract}

\section{Preliminaries}

In 1876, English astronomer and mathematician John Machin demonstrated an efficient method to compute digits of $\pi$ by using his famous discovery \cite{Beckmann1971,Berggren2004,Borwein2008,Agarwal2013}
\begin{equation}\label{OMF} % original Machin's formula
\frac{\pi}{4} = 4\arctan\left(\frac{1}{5}\right) - \arctan\left(\frac{1}{239}\right).
\end{equation}
In particular, due to relatively rapid convergence of this formula, he was the first to compute more than $100$ digits of $\pi$. Nowadays, the equations of kind
\[
\frac{\pi}{4} = \sum_{j = 1}^J {A_j\arctan\left(\frac{1}{B_j}\right)},
\]
where $A_j$ and $B_j$ are either integers or rational numbers, are named after him as the Machin-like formulas for $\pi$ \cite{Beckmann1971,Berggren2004,Borwein2008,Agarwal2013}.

Theorem~\ref{Theorem1.1} below shows the arctangent formula \eqref{AF} for $\pi$. We can use this equation as a starting point to generate the-two term Machin-like formula for $\pi$ of kind
\[
\frac{\pi}{4} = 2^{k - 1}\arctan\left(\frac{1}{x}\right) + \arctan\left(\frac{1}{y}\right),
\]
where $k,x \in \mathbb{N}$ and $y \in \mathbb{Q}$. Further, we will show the significance of the multiplier $2^{k - 1}$ in this formula.

\begin{theorem}\label{Theorem1.1}
There is a formula for $\pi$ at any integer $k \ge 1$ \cite{Abrarov2018}:
\begin{equation}\label{AF} % arctangent formula
\frac{\pi}{4} = 2^{k - 1}\arctan\left(\frac{\sqrt{2 - a_{k - 1}}}{a_k}\right),
\end{equation}
where $a_0 = 0$ and $a_k = \sqrt{2 + a_{k - 1}}$ are nested radicals.
\end{theorem}
\begin{proof}
Since
\[
\cos\left(\frac{\pi}{2^2}\right) = \frac{1}{2}\sqrt{2} = \frac{1}{2}{a_1},
\]
from the identity
\[
\cos\left(2x\right) = 2\cos^2\left(x\right) - 1,
\]
it follows, by induction, that
\[
\cos\left(\frac{\pi}{2^3}\right) = \frac{1}{2}\sqrt{2 + \sqrt{2}} = \frac{1}{2}{a_2},
\]
\[
\cos\left(\frac{\pi}{2^4}\right) = \frac{1}{2}\sqrt{2 + \sqrt{2 + \sqrt{2}}}  = \frac{1}{2}{a_3},
\]
\[
\cos\left(\frac{\pi}{2^5}\right) = \frac{1}{2}\sqrt{2 + \sqrt{2 + \sqrt{2 + \sqrt{2}}}}  = \frac{1}{2}{a_4},
\]
\[
\vdots\\
\]
\[
\cos\left(\frac{\pi}{2^{k + 1}}\right) = \frac{1}{2}\underbrace{\sqrt{2 + \sqrt{2 + \sqrt{2 + \cdots + \sqrt{2}}}}}_{k\,\,{\rm square\,\,roots}} = \frac{1}{2}{a_k}.
\]
Therefore, using
\[
\sin\left(\frac{\pi}{2^{k + 1}}\right) = \sqrt{1 - \cos^2\left(\frac{\pi}{2^{k + 1}}\right)},
\]
we obtain
\[
\frac{\pi}{2^{k + 1}} = \arctan\left(\frac{\sqrt{1 - \cos^2\left(\frac{\pi}{2^{k + 1}}\right)}}{\cos\left(\frac{\pi}{2^{k + 1}}\right)}\right) = \arctan\left(\frac{\sqrt{2 - a_{k - 1}}}{a_k}\right)
\]
and equation \eqref{AF} follows.
\end{proof}

Lemma~\ref{Lemma1.2} and its proof below show how equation \eqref{AF} is related to the well-known limit \eqref{SRF} for $\pi$.
\begin{lemma}\label{Lemma1.2}
There is a limit such that \cite{Servi2003}
\begin{equation}\label{SRF} % square roots formula
\pi  = \lim_{k\to\infty} 2^k\sqrt{2 - \underbrace{\sqrt{2 + \sqrt{2 + \sqrt{2 + \cdots + \sqrt{2}}}}}_{k - 1\,\,{\rm square\,\,roots}}}.
\end{equation}
\end{lemma}
\begin{proof}
Let
\[
\lim_{k\to\infty} a_k = x.
\]
Then, we can write
\[
\lim_{k\to\infty} a_k = \lim_{k\to\infty}\sqrt{2 + a_{k - 1}} = \sqrt{2 + \lim_{k\to\infty}{a_k}} 
\]
or
\[
x = \sqrt{2 + x}
\]
or 
\[
{x^2} - x - 2 = 0.
\]
Solving this quadratic equation leads to two solutions $x = 2$ and $x =  - 1$. Since $a_k$ cannot be negative, we came to conclusion that
\[
\lim_{k\to\infty} a_k = 2.
\]
Consequently, this gives
\[
\lim_{k\to\infty}{\sqrt{2 - a_{k - 1}}} = \sqrt{2 - \lim_{k\to\infty}a_{k - 1}} = \sqrt{2 - \lim_{k\to\infty}{a_k}} = 0.
\]

Since the formula \eqref{AF} is valid for any arbitrarily large $k$, we can write
\[
\frac{\pi}{4} = \lim_{k\to\infty} 2^{k - 1}\arctan\left(\frac{\sqrt{2 - a_{k - 1}}}{a_k}\right)
\]

As $\arctan\left(x\right)/x \to 1$ when $x\to 0$, we have
\[
\begin{aligned}
\frac{\pi}{4} &= \lim_{k\to\infty} 2^{k - 1}\frac{\sqrt{2 - a_{k - 1}}}{a_k} = \lim_{k\to\infty} {2^{k - 1}}\sqrt{2 - a_{k - 1}}/\lim_{k\to\infty}a_k \\
&= \lim_{k\to\infty}2^{k - 2}\sqrt{2 - a_{k - 1}} = \lim_{k\to\infty} 2^{k - 1}\sqrt{2 - a_k}
\end{aligned}
\]
and equation \eqref{SRF} follows. This completes the proof.
\end{proof}

Lemma~\ref{Lemma1.3} and its proof below show how equation \eqref{AF} can be transformed into the two-term Machin-like formula for $\pi$ \cite{Abrarov2017}.
\begin{lemma}\label{Lemma1.3}
In the following equation
\begin{equation}\label{TTMLF1} % two-term Machin-like formula 1
\frac{\pi}{4} = 2^{k - 1}\arctan\left(\frac{1}{\left\lfloor{a_k/\sqrt{2 - a_{k - 1}}}\right\rfloor}\right) + \arctan \left(\frac{1}{\beta_k}\right),
\end{equation}
the value $\beta_k$ is always a rational number.
\end{lemma}
\begin{proof}
It is convenient to define
\begin{equation}\label{AK} % alpha k
\alpha_k = \left\lfloor{a_k/\sqrt{2 - a_{k - 1}}}\right\rfloor.
\end{equation}
Using the identity
\[
\arctan\left(\frac{1}{x}\right) = \frac{1}{2i}\ln\left(\frac{x + i}{x - i}\right)
\]
and taking into consideration that
\[
\frac{\pi}{4} = \arctan\left(1\right),
\]
equation \eqref{TTMLF1} can be represented as
\[
\frac{\pi}{2}i = \ln\left(\frac{\alpha_k + i}{\alpha_k - i}\right)^{2^{k - 1}}\ln\left(\frac{\beta_k + i}{\beta_k - i}\right)
\]
or
\begin{equation}\label{R4I} % relation for i
i = \left(\frac{\alpha_k + i}{\alpha_k - i}\right)^{2^{k - 1}}\left(\frac{\beta_k + i}{\beta_k - i}\right).
\end{equation}

It is not difficult to see by substitution that the following formula \cite{Abrarov2017}
\begin{equation}\label{F4B1} % formula for beta 1
\beta_k = \frac{2}{\left(\frac{\alpha_k + i}{\alpha _k - i}\right)^{2^{k - 1}} - i} - i
\end{equation}
is a solution of the equation \eqref{R4I}. Since $k$ is a positive integer greater than or equal to $1$, we can see that the real and imaginary parts of the expression
\[
\left(\frac{\alpha_k + i}{\alpha_k - i}\right)^{2^{k - 1}} = \left(\frac{\alpha_k^2 - 1}{\alpha_k^2 + 1} + i\frac{2\alpha_k}{\alpha_k^2 + 1}\right)^{2^{k - 1}}
\]
cannot be an irrational number if $\alpha_k$ is an integer. This means that the real and imaginary parts of the value $\beta_k$ must be both rational. Since the first arctangent term of equation \eqref{TTMLF1} is a real number, the second arctangent term is also a real number. Therefore, we conclude that the imaginary part of the value $\beta_k$ is equal to zero. This completes the proof.
\end{proof}

This is not the only method to generate the two-term Machin-like formulas for $\pi$ of kind \eqref{TTMLF1}. Recently, Gasull {\it et al}. proposed a different method to derive the same equation (see Table 2 in \cite{Gasull2023}). Lemma~\ref{Lemma1.4} and its proof below show how the equation \eqref{TTMLF1} can be represented in a trigonometric form \cite{Abrarov2017}.

\begin{lemma}\label{Lemma1.4}
Equation \eqref{TTMLF1} can be expressed as
\begin{equation}\label{TTMLF2} % two-term Machin-like formula 2
\frac{\pi}{4} = 2^{k - 1}\arctan\left(\frac{1}{\alpha_k}\right) + \arctan\left(\frac{1 - \sin\left(2^{k - 1}\arctan\left(\frac{2\alpha_k}{\alpha_k^2 - 1}\right)\right)}{\cos\left(2^{k - 1}\arctan\left(\frac{2\alpha_k}{\alpha _k^2 - 1}\right) \right)}\right)
\end{equation}
\end{lemma}
\begin{proof}
Define
\[
\kappa_1 = \frac{\alpha_k^2 - 1}{\alpha_k^2 + 1}
\]
and
\[
\lambda_1 = \frac{2\alpha _k}{\alpha_k^2 + 1}
\]
such that
\begin{equation}\label{F4B2} % formula for beta 2
\beta_k = \frac{2}{\left(\kappa_1 + i\lambda_1\right)^{2^{k - 1}} - i} - i,
\end{equation}
in accordance with equation \eqref{F4B1}. Then, using de Moivre's formula we can write the complex number in polar form as
\[
\begin{aligned}
\left(\kappa_1 + i\lambda_1\right)^{2^{k - 1}} =& \left(\kappa_1^2 + \lambda_1^2\right)^{2^{k - 2}}\left(\cos\left(2^{k - 1}{\rm Arg}\left(\kappa_1 + i\lambda_1\right)\right)\right) \\
&+ i\sin\left(2^{k - 1}{\rm Arg}\left(\kappa_1 + i\lambda_1\right)\right).
\end{aligned}
\]
Thus, substituting this equation into equation \eqref{F4B2}, we obtain
\begin{equation}\label{F4B3} % formula for beta 3
\beta_k = \frac{\cos\left(2^{k - 1}{\rm Arg}\left(\kappa_1 + i\lambda_1\right)\right)}{1 - \sin\left(2^{k - 1}{\rm Arg}\left(\kappa_1 + i\lambda_1\right)\right)} = \frac{\cos\left(2^{k - 1}{\rm Arg}\left(\frac{\alpha_k + i}{\alpha_k - i}\right)\right)}{1 - \sin\left(2^{k - 1}{\rm Arg}\left( \frac{\alpha_k + i}{\alpha_k - i}\right)\right)}.
\end{equation}

Using the relation
\[
{\rm Arg}\left(x + iy\right) = \arctan\left(\frac{y}{x}\right), \qquad x > 0,
\]
we can write
\begin{equation}\label{I4AF} % identity for argument function
{\rm Arg}\left(\kappa_1 + i\lambda_1\right) = \arctan\left(\frac{2\alpha_k}{\alpha_k^2 - 1}\right).
\end{equation}
Consequently, from the identities \eqref{F4B3} and \eqref{I4AF}, we get
\begin{equation}\label{F4B4} % formula for beta 4
\beta_k = \frac{\cos\left(2^{k - 1}\arctan\left(\frac{2\alpha_k}{\alpha_k^2 - 1}\right)\right)}{1 - \sin\left(2^{k - 1}\arctan\left(\frac{2\alpha_k}{\alpha_k^2 - 1}\right)\right)}
\end{equation}
and equation \eqref{TTMLF2} follows.
\end{proof}

It should be noted that computation of the constant $\beta_k$ by using equation \eqref{F4B1} is not optimal. Specifically, at lager values of the integer $k$ its application slows down the computation. Application of the equation \eqref{F4B4} for computation of the constant $\beta_k$ is also not desirable due to presence of the trigonometric functions. Theorem~\ref{Theorem1.5} and its proof show how this problem can be effectively resolved \cite{Abrarov2017}.
\begin{theorem}\label{Theorem1.5}
The rational number $\beta_k$ is given by
\begin{equation}\label{F4B5} % formula for beta 5
\beta_k = \frac{\kappa_k}{1 - \lambda_k},
\end{equation}
where the coefficients ${\kappa_k}$ and ${\lambda_k}$ can be found by a two-step iteration
\begin{equation}\label{TSI} % two-step iteration
\left\{
\begin{aligned}
\kappa_n = \kappa_{n - 1}^2 - \lambda_{n - 1}^2 \\
\lambda_n = 2\kappa_{n - 1}\lambda_{n - 1},
\end{aligned}
\right. \qquad n = 2,3,4,\ldots,k
\end{equation}
with initial values
\[
\kappa_1 = \frac{\alpha_k^2 - 1}{\alpha_k^2 + 1}
\]
and
\[
\lambda_1 = \frac{2\alpha_k}{\alpha_k^2 + 1}.
\]
\end{theorem}

\begin{proof}
We notice that the following power reduction
\[
\begin{aligned}
\left(\kappa_1 + i\lambda_1\right)^{2k - 1} &= \overbrace{\left(\left(\left(\kappa_1 + i\lambda_1\right)^2\right)^{2\,\,\cdots} \right)^2}^{{k - 1}\,\,{\rm powers\,\,of\,\,2}} = \overbrace{\left(\left(\left(\kappa_2 + i\lambda_2 \right)^2\right)^{2\,\,\cdots}\right)^2}^{{k - 2}\,\,{\rm powers\,\,of\,\,2}} \\
&= \overbrace{\left(\left(\left(\kappa_3 + i\lambda_3\right)^2\right)^{2\,\,\cdots}\right)^2}^{{k - 3}\,\,{\rm powers\,\,of\,\,2}} = \overbrace{\left(\left(\left(\kappa_n + i\lambda_n\right)^2\right)^{2\,\,\cdots}\right)^2}^{{k - n}\,\,{\rm powers\,\,of\,\,2}} \\
&= \left(\left(\kappa_{k - 2} + i\lambda_{k - 2}\right)^2\right)^2 = \left(\kappa_{k - 1} + i\lambda_{k - 1}\right)^2 = \kappa_k + i\lambda_k,
\end{aligned}
\]
where the numbers $\kappa_n$ and $\lambda_n$ can be found by two-step iteration \eqref{TSI}, leads to
\[
\beta_k = \frac{2}{\kappa_k - i\lambda_k - i} - i = \frac{2\kappa_k}{\kappa_k^2 + \left(\lambda_k - 1\right)^2} + i\left(\frac{2\left(1 - \lambda_k\right)}{\kappa_k^2 + \left(\lambda_k - 1\right)^2} - 1\right),
\]
according to the equation \eqref{F4B1}.

From the Lemma~\ref{Lemma1.3} it follows that the imaginary part of the value $\beta_k$ is equal to zero. Consequently, the equation above can be simplified as
\begin{equation}\label{F4B6} % formula for beta 6
\beta_k = \frac{2\kappa_k}{\kappa_k^2 + \left(\lambda_k - 1\right)^2}.
\end{equation}
However, since
\[
i\left(\frac{2\left(1 - \lambda_k \right)}{\kappa_k^2 + \left(\lambda_k - 1\right)^2} - 1\right) = 0
\]
it follows that
\[
\frac{2\left(1 - \lambda_k\right)}{\kappa_k^2 + \left(\lambda_k - 1\right)^2} = 1 \Leftrightarrow \kappa_k^2 + \left(\lambda_k - 1\right)^2 = 2\left(1 - \lambda_k\right).
\]
Substituting this result into the denominator of equation \eqref{F4B6}, we obtain equation \eqref{F4B5}. This completes the proof.
\end{proof}

There is an interesting relation between $\beta_k$ and nested radicals $a_{k - 1}$ and $a_k$. Specifically, comparing equation \eqref{TTMLF1} with equation \eqref{F4P2} from the Lemma~\ref{Lemma1.6} below, one can see that
\[
\arctan\left(\frac{1}{\beta_k}\right) = -2^{k - 1}\arctan\left(\frac{\left\{a_k/\sqrt{2 - {a_{k - 1}}} \right\}}{1 + \left\lfloor a_k/\sqrt{2 - a_{k - 1}} \right\rfloor\left( a_k/\sqrt{2 - a_{k - 1}}\right)}\right),
\]
where
\[
\left\{a_k/\sqrt{2 - a_{k - 1}}\right\} = a_k/\sqrt{2 - a_{k - 1}} - \left\lfloor a_k/\sqrt{2 - a_{k - 1}}\right\rfloor
\]
denotes the fractional part.
\begin{lemma}\label{Lemma1.6}
\[
\lim_{k\to\infty} \arctan\left(\frac{1}{\beta_k}\right) = 0.
\]
\end{lemma}
\begin{proof}
It is not difficult to see that using change of the variable
\[
y\to\frac{y}{1 + \left(x + y\right)x}
\]
in the trigonometric identity
\[
\arctan\left(x\right) + \arctan\left(y\right) = \arctan\left(\frac{x + y}{1 - xy}\right)
\]
leads to
\begin{equation}\label{AI} % arctangent identity
\arctan\left(x + y\right) = \arctan\left(x\right) + \arctan\left(\frac{y}{1 + \left(x + y\right)x} \right).
\end{equation}

Define for convenience the fractional part as
\[
\gamma_k = \left\{a_k/\sqrt{2 - a_{k - 1}}\right\}.
\]
Then, from equation \eqref{AF} it follows that
\[
\frac{\pi}{4} = 2^{k - 1}\arctan\left(\frac{1}{\alpha_k + \gamma_k}\right)
\]
Solving the equation
\[
\frac{1}{\alpha_k + \gamma_k} = \frac{1}{\alpha_k} + z,
\]
we get
\[
z = -\frac{\gamma_k}{\alpha_k\left(\alpha_k + \gamma_k\right)}.
\]
Therefore, we have	
\begin{equation}\label{F4P1} % formula for pi 1
\frac{\pi}{4} = 2^{k - 1}\arctan\left(\frac{1}{\alpha_k} - \frac{\gamma_k}{\alpha_k\left(\alpha_k +\gamma_k\right)}\right).
\end{equation}

Using the identity \eqref{AI}, we can rewrite equation \eqref{F4P1} as
\[
\frac{\pi}{4} = 2^{k - 1}\left(\arctan\left(\frac{1}{\alpha_k}\right) - \arctan\left(\frac{\gamma_k}{1 +\alpha_k\left(\alpha_k + \gamma\right)}\right)\right)
\]
or
\begin{equation}\label{F4P2}  % formula for pi 2
\begin{aligned}
\frac{\pi}{4} &= 2^{k - 1}\left(\arctan\left(\frac{1}{\left\lfloor a_k/\sqrt{2 - a_{k - 1}} \right\rfloor}\right)\right. \\
&\left. -\arctan \left( \frac{\left\{a_k/\sqrt{2 - a_{k - 1}}\right\}}{1 + \left\lfloor a_k/\sqrt{2 - a_{k - 1}}\right\rfloor \left(a_k/\sqrt{2 - a_{k - 1}}\right)}\right)\right).
\end{aligned}
\end{equation}

As the fractional part
\[
\left\{a_k/\sqrt{2 - a_{k - 1}}\right\} < 1
\]
while
\[
\lim_{k\to\infty} \left(1 + \left\lfloor a_k/\sqrt{2 - a_{k - 1}}\right\rfloor\left(a_k/\sqrt{2 - a_{k - 1}}\right)\right) = \infty,
\]
we conclude that
\[
\begin{aligned}
\lim_{k\to\infty} \,\, &-2^{k-1}\arctan\left( \frac{\left\{ a_k/\sqrt {2 - a_{k - 1}} \right\}}{1 + \left\lfloor a_k/\sqrt{2 - a_{k - 1}}\right\rfloor\left(a_k/\sqrt {2 - a_{k - 1}}\right)}\right) \\
&= \lim_{k\to\infty}\arctan\left( \frac{1 - \lambda_k}{\kappa_k}\right) = \lim_{k\to\infty}\arctan\left(\frac{1}{\beta_k}\right) = 0.
\end{aligned}
\]
\end{proof}

The Lemma~\eqref{Lemma1.7} and its proof below show how to obtain a single-term rational approximation \eqref{STRA} for $\pi$ by truncating equation \eqref{L4B}.
\begin{lemma}\label{Lemma1.7}
\begin{equation}\label{L4B} % limit for beta
\frac{\pi}{4} = \lim_{k\to\infty} \frac{2^{k - 1}}{\left\lfloor a_k/\sqrt{2 - a_{k - 1}}\right\rfloor}.
\end{equation}
\end{lemma}
\begin{proof}
From the Lemma~\ref{Lemma1.6} it immediately follows that
\begin{equation}\label{L4AF} % limit for arctangent function
\frac{\pi}{4} = \lim_{k\to\infty} 2^{k - 1}\arctan\left(\frac{1}{\left\lfloor a_k/\sqrt{2 - a_{k - 1}}\right\rfloor}\right).
\end{equation}
According to Lemma~\eqref{Lemma1.2}
\[
\lim_{k\to\infty}\sqrt{2 - a_{k - 1}} = 0
\]
and
\[
\lim_{k\to\infty} a_k = 2.
\]
Therefore, we can infer that
\[
\lim_{k\to\infty}\left\lfloor a_k/\sqrt{2 - a_{k - 1}}\right\rfloor = \infty
\]
or
\[
\lim_{k\to\infty} \frac{1}{\left\lfloor a_k/\sqrt{2 - a_{k - 1}}\right\rfloor} = 0.
\]
This limit implies that the argument of the arctangent function in equation \eqref{L4AF} tends to zero with increasing the integer $k$. Consequently, the limit \eqref{L4B} follows from the limit \eqref{L4AF} and the proof is completed.
\end{proof}

Motivated by recent publications \cite{Gasull2023,WolframCloud,Campbell2023,Maritz2023,Spichal2024,Alferov2024} in connection to our works \cite{Abrarov2017,Abrarov2018,Abrarov2022}, we proposed a methodology for determination of the coefficients $\alpha_k$ without computing nested square roots of $2$ (see equation \eqref{AK}) and developed an algorithm providing squared convergence per iteration in computing the digits of $\pi$. To the best of our knowledge, this approach is new and have never been reported.

\section{Methodologies}

\subsection{Arctangent function}

There are different efficient methods to compute the arctangent function in the Machin-like formulas. We will consider a few of them.

In our previous publication \cite{Abrarov2021a}, we show how the two-term Machin-like formula for $\pi $ represented in form \eqref{TTMLF2} is used to derive an iterative formula
\begin{equation}\label{IF} % iterative formula
\theta_{n,k} = \frac{1}{\frac{1}{\theta_{n - 1,k}} + \frac{1}{2^k}\left( 1 - \tan\left( \frac{2^{k - 1}}{\theta_{n - 1,k}}\right)\right)}, \qquad k \ge 1,
\end{equation}
where initial value can be taken as
\[
\theta_{1,k} = 2^{-k}
\]
such that
\begin{equation}\label{L4IF} % limit for iterative formula
\frac{\pi}{4} = 2^{k - 1}\lim_{n\to\infty} \frac{1}{\theta_{n,k}}.
\end{equation}
This iterative formula provides a squared convergence in computing digits of $\pi $ (see Mathematica code in \cite{Abrarov2021a}).

Taking change of the variable $\theta_n \to 1/\theta_n$ in equation \eqref{IF} yields a more convenient form
\begin{equation}\label{L4MIF} % limit for modified iterative formula
\theta_{n,k} = \theta_{n - 1,k} + 2^{-k}\left( 1 - \tan\left(2^{k - 1}\theta_{n - 1,k}\right)\right), \qquad k \ge 1.
\end{equation}
Consequently, the limit \eqref{L4IF} can be rearranged now as
$$
\frac{\pi}{4} = 2^{k - 1}\lim_{n\to\infty} \theta_{n,k}.
$$

Comparing this limit with equation \eqref{AF}, we can see that this iteration procedure results in
$$
\lim_{n\to\infty} \theta_{n,k} = \arctan\left(\frac{1}{\alpha_k}\right)
$$
and is used {\it de facto} for determination of the arctangent function. The detailed procedure showing how to implement the computation with high convergence rate is given in our recent publication \cite{Abrarov2024}.

Alternatively, we can transform the two-term into multi-term Machin-like formulas for $\pi$ consisting of only the integer reciprocals. In order to do this, we can use the following equation \cite{Abrarov2022}
\begin{equation}\label{T4MLF} % template for Machin-like fomulas
\frac{\pi}{4} = 2^{k - 1}\arctan\left(\frac{1}{\alpha_k}\right) + \left(\sum_{m = 1}^M \arctan \left( \frac{1}{\left\lfloor\mu_{m,k}\right\rfloor}\right)\right) + \arctan\left( \frac{1}{\mu_{M + 1,k}}\right),
\end{equation}
where
$$
\mu_{m,k} = \frac{1 + \left\lfloor\mu_{m - 1,k}\right\rfloor\mu_{m - 1,k}}{\left\lfloor\mu_{m - 1,k} \right\rfloor - \mu_{m - 1,k}}
$$
with an initial value
$$
\mu_{1,k} = \beta_k.
$$

For example, by taking $k = 2$ in the equation \eqref{T4MLF}, we can construct
\begin{equation}\label{HF4P} % Hermann's formula for pi
\begin{aligned}
\frac{\pi}{4} &= 2^{2 - 1}\arctan \left(\frac{1}{\alpha_2}\right) + \arctan\left(\frac{1}{\mu_{1,2}} \right) \\
&= 2\arctan\left(\frac{1}{2}\right) - \arctan\left(\frac{1}{7}\right).
\end{aligned}
\end{equation}
This equation is commonly known as the Hermann's formula for $\pi$ \cite{Borwein1987,Chien-Lih1997}.

By taking $k = 3$, the equation \eqref{T4MLF} gives
\[
\frac{\pi}{4} = 2^{3 - 1}\arctan\left(\frac{1}{\alpha_3}\right) - \arctan\left(\frac{1}{\mu_{1,3}}\right)
\]
representing the original Machin's formula \eqref{OMF} for $\pi$. 

The case $k = 4$ requires more calculations to obtain the multi-term formula for $\pi$ consisting of only integer reciprocals. In particular, at $M = 1$ we get
\[
\begin{aligned}
\frac{\pi}{4} &= 2^{4 - 1}\arctan\left(\frac{1}{\alpha_4}\right) + \arctan\left(\frac{1}{\mu_{1,4}} \right) \\
&= 8\arctan\left(\frac{1}{10}\right) - \arctan\left(\frac{1758719}{147153121}\right)
\end{aligned}
\]
At $M = 2$ and $M = 3$ equation \eqref{T4MLF} yields
\[
\begin{aligned}
\frac{\pi}{4} &= 2^{4 - 1}\arctan\left(\frac{1}{\alpha _4}\right) + \arctan\left( \frac{1}{\left\lfloor \mu_{1,4}\right\rfloor}\right) + \arctan \left(\frac{1}{\mu _{2,4}}\right)\\
&= 8\arctan\left(\frac{1}{10}\right) - \arctan\left(\frac{1}{84}\right)-\arctan\left(\frac{579275}{12362620883}\right)
\end{aligned}
\]
and
\[
\begin{aligned}
\frac{\pi}{4} &= 2^{4 - 1}\arctan\left(\frac{1}{\alpha_4}\right) + \arctan\left(\frac{1}{\left\lfloor \mu_{1,4}\right\rfloor}\right) + \arctan\left( \frac{1}{\left\lfloor\mu_{2,4}\right\rfloor}\right) + \arctan\left(\frac{1}{\mu_{3,4}}\right)\\
&= 8\arctan\left(\frac{1}{10}\right) - \arctan\left(\frac{1}{84}\right) - \arctan\left(\frac{1}{21342}\right) \\
&-\arctan\left(\frac{266167}{263843055464261}\right),
\end{aligned}
\]
respectively.

Repeating this procedure over and over again, at $M = 5$ we end up with 7-term Machin-like formula for $\pi$ consisting of only integer reciprocals
\begin{equation}
\begin{aligned}\label{STMLF} % seven-term Machin-like fomula
\frac{\pi}{4} &= 2^{4 - 1}\arctan\left(\frac{1}{\alpha_4}\right) + \left(\sum_{m = 1}^5\arctan\left( \frac{1}{\left\lfloor \mu_{m,4}\right\rfloor}\right)\right) + \arctan\left(\frac{1}{\mu_{6,4}}\right)\\
&= 8\arctan\left(\frac{1}{10}\right) - \arctan\left(\frac{1}{84}\right) - \arctan\left(\frac{1}{21342}\right) \\
&-\arctan\left(\frac{1}{991268848}\right) - \arctan\left(\frac{1}{193018008592515208050}\right) \\
&-\arctan\left(\frac{1}{197967899896401851763240424238758988350338}\right) \\
&-\arctan\left(\frac{1}{\Omega}\right),
\end{aligned}
\end{equation}
where
\[
\begin{aligned}
\Omega  = &\,11757386816817535293027775284419412676799191500853701 \ldots \\
&8836932014293678271636885792397,
\end{aligned}
\]
is the largest integer (see the Mathematica code in \cite{Abrarov2024} that validates this seven-term Machin-like formula for $\pi$).

As we can see, Hermann's \eqref{HF4P}, Machin's \eqref{OMF} and derived \eqref{STMLF} formulas for $\pi$ belong to the same generic group as all of them can be constructed from the same equation-template \eqref{T4MLF}.

It is not difficult to show that applying the following identity
\[
\arctan\left(\frac{1}{x}\right) = \arctan\left(\frac{1}{x + 1}\right) + \arctan\left(\frac{1}{x^2 + x + 1}\right), \qquad x\notin \left[-1,0\right],
\]
to the equation \eqref{STMLF}, we obtain
\[
\begin{aligned}
\frac{\pi}{4} & = 8\arctan\left(\frac{1}{10}\right) - \arctan\left(\frac{1}{84}\right) - \arctan\left(\frac{1}{21342}\right) \\
&-\arctan\left(\frac{1}{991268848}\right) - \arctan\left(\frac{1}{193018008592515208050}\right) \\
&-\arctan\left(\frac{1}{197967899896401851763240424238758988350338}\right) \\
&-\sum_{n = 1}^N \arctan\left(\frac{1}{\omega_n}\right), \qquad N \ge 2,
\end{aligned}
\]
where
\[
\begin{aligned}
&\omega_1 = \Omega + 1, \\
&\omega_2 = \left(\omega_1 - 1\right)^2 + \omega_1 + 1, \\
&\omega_3 = \left(\omega_2 - 1\right)^2 + \omega_2 + 1, \\
&\vdots \\
&\omega_{N - 1} = \left(\omega_{N - 2} - 1\right)^2 + \omega_{N - 2} + 1, \\
&\omega_N = \left(\omega_{N - 1} - 1\right)^2 + \omega_{N - 1}.
\end{aligned}
\]
This expanded variation of equation \eqref{STMLF} together with iterative formula \eqref{L4MIF} can also be used for computing digits of $\pi$ (see section 5 in \cite{Abrarov2024} for more details).

Our empirical results show that two arctangent series expansions can be used for computation with rapid convergence. The first equation is Euler's expansion series given by \cite{Chien-Lih2005}
\begin{equation}\label{EF4AF} % Euler's formula for arctangent function
\arctan\left(x\right) = \sum_{n = 0}^\infty\frac{2^{2n}{\left(n!\right)^2}}{\left(2n + 1\right)!}\frac{x^{2n + 1}}{\left(1 + x^2\right)^{n + 1}}.
\end{equation}

The second equation is \cite{Abrarov2017,Abrarov2023}
\begin{equation}\label{SIF4AF} % semi-iterative formula for arctangent function
\arctan\left(x\right) = 2\sum_{n = 1}^\infty  \frac{1}{2n - 1}\frac{g_n\left(x\right)}{g_n^2\left(x\right) + h_n^2\left(x\right)},
\end{equation}
where
$$
g_n\left(x\right) = g_{n - 1}\left(x\right)\left(1 - 4/x^2\right) + 4h_{n - 1}\left(x\right)/x
$$
and
$$
h_n\left(x\right) = h_{n - 1}\left(x\right)\left(1 - 4/x^2\right) - 4g_{n - 1}\left(x\right)/x
$$
with initial values
$$
g_1\left(x\right) = 2/x
$$
and
$$
h_1\left(x\right) = 1.
$$

Computational test we performed shows that equation \eqref{SIF4AF} is faster in convergence than equation \eqref{EF4AF}. Recently, we proposed the generalization of the arctangent series expansion \eqref{SIF4AF} \cite{Abrarov2023}.

\subsection{Tangent function}

Generally, transformation of the two-terms to multi-terms formulas for $\pi $ with integer reciprocals is not required. In particular, we can use Newton–Raphson iteration method. For example, both arctangent terms in the two-term Machin-like formula \eqref{TTMLF1} for $\pi$ can be computed directly by using the following iterative formulas
$$
\sigma_n = \sigma_{n - 1} - \left(1 - \left(\frac{2\tan \left(\sigma_{n - 1}/2\right)}{1 + \tan^2\left(\sigma_{n - 1}/2\right)}\right)^2\right)\left( \frac{2\tan\left(\sigma_{n - 1}/2\right)}{1 - {{\tan }^2}\left(\sigma_{n - 1}/2\right)} - \frac{1}{\alpha_k}\right)
$$
and
$$
\tau_n = \tau_{n - 1} - \left(1 - \left(\frac{2\tan \left(\tau_{n - 1}/2\right)}{1 + \tan^2\left(\tau_{n - 1}/2\right)}\right)^2\right)\left( \frac{2\tan\left(\tau_{n - 1}/2\right)}{1 - {{\tan }^2}\left(\tau_{n - 1}/2\right)} - \frac{1}{\beta_k}\right)
$$
with initial values
$$
\sigma_1 = \frac{1}{\alpha_k}
$$
and
$$
\tau_1 = \frac{1}{\beta_k}
$$
such that
$$
\arctan\left(\frac{1}{\alpha_k}\right) = \lim_{n\to\infty} \sigma_n
$$
and
$$
\arctan\left(\frac{1}{\beta_k}\right) = \lim_{n\to\infty} \tau_n.
$$

Since the convergence of the Newton–Raphson iteration is quadratic \cite{Ypma1995}, with proper implementation of the tangent function we may achieve an efficient computation.

The tangent function can be expanded as
\begin{equation}\label{TFSE} % tangent function series expansion
\begin{aligned}
\tan\left(x\right) &= \sum_{n = 1}^\infty  \frac{\left(- 1\right)^{n - 1}{2^{2n}}\left(2^{2n} - 1 \right) B_{2n}}{\left(2n\right)!}x^{2n - 1} \\
&= x + \frac{x^3}{3} + \frac{2x^5}{15} + \frac{17x^7}{315} + \frac{62x^9}{2835} +  \cdots \,,
\end{aligned}
\end{equation}
where $B_{2n}$ are the Bernoulli numbers that can be defined by the exponential generating function
\[
\frac{x}{e^x - 1} = \sum_{n \ge 0} \frac{B_n x^n}{n!}.
\]
However, application of equation \eqref{TFSE} is not desirable since the computation of the Bernoulli numbers $B_{2n}$ itself is a big challenge \cite{Knuth1967,Harvey2010,Bailey2013,Beebe2017}.

In order to resolve this problem, we proposed the following limit \cite{Abrarov2024}
\begin{equation}\label{L4TF} % limit for tangent function
\tan\left(x\right) = \lim_{n\to\infty} \frac{2p_n^2\left(x\right)}{q_n\left(x\right)},
\end{equation}
where
\[
p_n\left(x\right) = p_{n - 1}\left(x\right) + r_{n - 1}\left(x\right),
\]
\[
q_n\left(x\right) = q_{n - 1}\left(x\right) + 2^{2n - 1}r_{n - 1}\left(x\right),
\]
such that 
\[
r_n = \frac{\left(-1\right)^n}{\left(2n + 1\right)!}x^{2n + 1},
\]
with initial values
\[
p_0\left(x\right) = 0,
\]
\[
q_0\left(x\right) = 0.
\]
Specifically, it has been shown that at $k = 27$ in the two-term Machin-like formula \eqref{TTMLF1} for $\pi$, application of equation \eqref{L4TF} results in more than $17$ digits of $\pi$ per increment $n$ (see the Mathematica codes in \cite{Abrarov2024}).

Since the multiplier $x^{2n + 1}$ in equation \eqref{L4TF} increases exponentially with increasing integer $n$, the iteration process rapidly reduces number of the digits in the mantissa when the tangent function with some fixed accuracy is computed. Decreasing number of digits in the mantissa may be advantageous as the computation requires less memory usage.

\section{Algorithmic implementation}

In our recent publication we have shown that \cite{Abrarov2021a}
\[
\beta_k \approx \frac{2}{1 - \tan\left(\frac{2^{k - 1}}{\alpha_k}\right)}.
\]
Consequently, in accordance with equation \eqref{TTMLF1}, we obtain the following approximation
\[
\frac{\pi}{4} \approx 2^{k - 1}\arctan\left(\frac{1}{\alpha_k}\right) + \frac{1}{2}\left(1 - \tan\left( \frac{2^{k - 1}}{\alpha_k}\right)\right).
\]

At sufficiently large $k$ the value $1/\alpha_k << 1$. Therefore, according to the the Lemma~\ref{Lemma1.7}, in this equation we can replace the first arctangent term by a rational number $2^{k - 1}/\alpha_k$. This gives
\[
\frac{\pi}{4} \approx \frac{2^{k - 1}}{\alpha_k} + \frac{1}{2}\left( 1 - \tan \left( \frac{2^{k - 1}}{\alpha_k}\right)\right).
\]

Unfortunately, we cannot compute efficiently the tangent function in this approximation since its argument $2^{k - 1}/\alpha_k$ is not a small number as it tends to $\pi /4$ with increasing $k$. However, again by taking into consideration that $1/\alpha_k << 1$ from which it follows that
\[
\frac{1}{\alpha_k} \approx \arctan\left(\frac{1}{\alpha_k}\right),
\]
we can write
\[
\frac{\pi}{4} \approx \frac{2^{k - 1}}{\alpha_k} + \frac{1}{2}\left( 1 - \tan\left( 2^{k - 1}\arctan \left(\frac{1}{\alpha_k}\right)\right)\right).
\]
Now we can take advantage from the fact that the multiplier $2^{k - 1}$ is continuously divisible by $2$. Therefore, we can use the trigonometric identity
\[
\tan\left(2\arctan\left(x\right)\right) = \frac{2x}{1 - x^2}
\]
$k - 1$ times over and over again. Thus, this leads to the following iterative formula
\begin{equation}\label{EF} % eta function
\eta_n\left(x\right) = \frac{2\eta _{n - 1}\left(x\right)}{1 - \eta_{n - 1}^2\left(x\right)}
\end{equation}
with an initial value
\[
\eta_1\left(x\right) = \frac{2x}{1 - x^2}
\]
such that
\[
\eta_{k - 1}\left(\frac{1}{\alpha_k}\right) = \tan\left(2^{k - 1}\arctan\left(\frac{1}{\alpha_k}\right)\right).
\]

Since the left side of the equation above provides an exact value without (tangent and arctangent) trigonometric functions, we can regard this equation
\begin{equation}\label{TTRA} % two-term rational approximation
\frac{\pi}{4} \approx \frac{2^{k - 1}}{\alpha_k} + \frac{1}{2}\left(1 - \eta_{k - 1}\left(\frac{1}{\alpha_k}\right)\right)
\end{equation}
as a rational approximation.

This rational approximation of the two-term Machin-like formula for $\pi$ can be used in an algorithm providing a quadratic convergence. This can be achieved with help of the Theorem~\ref{Theorem3.1}.

\begin{theorem}\label{Theorem3.1}
There is a conversion formula
\begin{equation}\label{BFF4P} % binary form formula for pi
\lim_{k\to\infty} \sum\limits_{n = 1}^k{\frac{1}{10^{n + 1}}\left(\alpha_n\,\bmod 2\right)} = {0.01010001011111001100 \ldots}\,\,,
\end{equation}
such that
\[
{\left[0.01010001011111001100 \ldots\right]_2} = \frac{1}{\pi}.
\]
\end{theorem}
\begin{proof}
The proof is related to the parity of the integer $\alpha_k$. According to the Lemma~\ref{Lemma1.7}, we can write
\begin{equation}\label{STRA} % single-term rational approximation
\frac{\pi}{4} \approx \frac{2^{k - 1}}{\alpha_k}
\end{equation}
or
$$
\frac{1}{\pi} \approx \frac{\alpha_k}{4\times 2^{k - 1}}.
$$
Consequently, if the integer $\alpha_k$ is even, then
$$
\frac{\alpha_k}{4\times 2^{k - 1}} - \frac{\alpha_{k - 1}}{4\times 2^{k - 2}} = 0_2.
$$
However, if the integer $\alpha_k$ is odd, then
$$
\frac{\alpha_k}{4 \times 2^{k - 1}} - \frac{\alpha_{k - 1}}{4 \times 2^{k - 2}} = \underbrace{0.00 \ldots 00}_{k + 1\,{\rm{zeros}}}{1_2}.
$$
This means that $\alpha_k$ contributes a binary digit $\underbrace{0.00 \ldots 00}_{k + 1\,{\rm{zeros}}}{1_2}$ to the previous value $\alpha_{k - 1}$ if and only if $\alpha_k$ is odd. This completes the proof.
\end{proof}

Consider an example. There are four consecutive values $\alpha_4 = 10$, $\alpha_5 = 20$ and $\alpha_6 = 40$ and $\alpha_7 = 81$. Since the first three values are even, we have
$$
\frac{\alpha_4}{4\times 2^{4 - 1}} = \frac{\alpha_5}{4\times 2^{5 - 1}} = \frac{\alpha_6}{4\times 2^{6 - 1}} = 0.0101_2.
$$
However, since $\alpha_7 = 81$ is odd, we obtain
$$
\frac{\alpha_7}{4\times 2^{7 - 1}} = 0.0101_2 + 0.00000001_2 = 0.01010001_2.
$$

Consider how number of digits of $\pi$ can be doubled without computing square roots for the nested radicals ${a_k}$. We can take, for example, $k = 7$. This yields
\begin{equation}\label{CK7} % case k7
\begin{aligned}
\alpha_7 &= \left\lfloor\frac{a_7}{\sqrt{2 - a_{7 - 1}}}\right\rfloor \\
&= \left\lfloor\frac{\sqrt {2 + \sqrt{2 + \sqrt{2 + \sqrt{2 + \sqrt{2 + \sqrt{2 + \sqrt 2}}}}}}}{\sqrt{2 - \sqrt{2 + \sqrt{2 + \sqrt{2 + \sqrt{2 + \sqrt{2 + \sqrt 2}}}}}}} \right\rfloor = 81.
\end{aligned}
\end{equation}
However, it is not reasonable to compute the square roots of $2$ so many times to obtain this number. Instead, we can simplify computation considerably by using the value $1/\pi$ in the binary form according to the Theorem~\ref{Theorem3.1}. Thus, ignoring the first two initial zeros in the binary output of the equation \eqref{BFF4P}, we have a corresponding sequence
\begin{equation}\label{S4BO} % sequence for binary output
(s)_1^\infty=\left(1,0,1,0,0,0,1,0,1,1,1,1,1,0,0,\ldots\right).
\end{equation}
This sequence can be obtained by using the built-in Mathematica directly using $1/\pi$. For example, the following code:
\small
\begin{verbatim}
RealDigits[
  ImportString[ToString[BaseForm[N[1/\[Pi],20],2]], 
    "Table"][[1]][[1]]][[1]][[1;;20]]
\end{verbatim}
\normalsize
returns the first $20$ digits from the sequence \eqref{S4BO}:
\small
\begin{verbatim}
{1, 0, 1, 0, 0, 0, 1, 0, 1, 1, 1, 1, 1, 0, 0, 1, 1, 0, 0, 0}
\end{verbatim}
\normalsize

From this sequence, we choose the sub-sequence (say, up to seventh element)
\[
(s)_1^7=\left(1,0,1,0,0,0,1\right)
\]
and apply it accordingly as
\begin{equation*}\label{SBSP} % step-by-step procedure
{\alpha _k} = \left\{
\begin{aligned}
&2{\alpha _{k - 1}}, &\quad &{\rm if\,{\it k}^{th}\,binary\,digit\,is\,0} \\
&2{\alpha _{k - 1}} + 1, &\quad &{\rm if\,{\it k}^{th}\,binary\,digit\,is\,1}.
\end{aligned}
\right.
\end{equation*}
Explicitly, defining $\alpha_0 = 0$ this step-by-step procedure results in
\[
\begin{aligned}
{\alpha _1} &= 2{\alpha _0} + 1 = \left( {2 \times 0 + 1} \right)  = 1, \\
{\alpha _2} &= 2{\alpha _1} + 0 = \left( {2 \times 1 + 0} \right)  = 2, \\
{\alpha _3} &= 2{\alpha _2} + 1 = \left( {2 \times 2 + 1} \right)  = 5, \\
{\alpha _4} &= 2{\alpha _3} + 0 = \left( {2 \times 5 + 0} \right)  = 10, \\
{\alpha _5} &= 2{\alpha _4} + 0 = \left( {2 \times 10 + 0} \right) = 20, \\
{\alpha _6} &= 2{\alpha _5} + 0 = \left( {2 \times 20 + 0} \right) = 40, \\
{\alpha _7} &= 2{\alpha _6} + 1 = \left( {2 \times 40 + 1} \right) = 81.
\end{aligned}
\]
Thus, we can see how a very simple procedure can be used to determine the value of the rational number $\alpha_7 = 81$ without using a sophisticated equation \eqref{CK7} consisting of $14$ nested square roots of $2$.

At $\alpha_7 = 81$ the corresponding Machin-like formula is
\begin{equation}\label{MLF4K7} % Machin-like formula for k7
\begin{aligned}
\frac{\pi}{4} &= 2^{7 - 1}\arctan\left(\frac{1}{\alpha_7}\right) + \arctan\left(\frac{1}{\beta_7}\right)\\
&= 64\arctan\left(\frac{1}{81}\right) - \arctan\left( \frac{\overbrace {2154947582 \ldots 4298183679}^{111\,\,{\rm{digits}}}}{\underbrace{4599489202 \ldots 6981324801}_{113\,\,{\rm{digits}}}}\right),
\end{aligned}
\end{equation}
where the constant
\[
\beta _7 =  - \frac{\overbrace{4599489202 \ldots 6981324801}^{113\,\,{\rm{digits}}}}{\underbrace {2154947582 \ldots 4298183679}_{111\,\,{\rm{digits}}}}
\]
can be computed either by using equation \eqref{F4B1} or, more efficiently, by using equation \eqref{F4B5} based on two-step iteration \eqref{TSI}.

The following Mathematica code:
\small
\begin{verbatim}
(* String for long number \[Beta]_7 *)
strBeta7=
  ToString[StringJoin[
    "21549475820057881611210311984288158234143531212163819254",
      "1568712000964806160594022446140062110943660584298183679/",
        "459948920218008069525744651226752553899687099736076594466",
          "78719072620659988130828378620624183170066256006981324801"
            ]];
						
(* Verification *)
Print[\[Pi]/4==64*ArcTan[1/81]-ArcTan[ToExpression[strBeta7]]];
\end{verbatim}
\normalsize
validates the equation \eqref{MLF4K7} by returning {\bf{\ttfamily True}}.

Suppose that we do not know the sequence other than $\left(s\right)_1^7$. However, with help of the equation \eqref{TTRA} we can find other digits of $\pi$ in the iterative process. In particular, using equation \eqref{EF}, we have
$$
\eta_{7 - 1}\left(\frac{1}{81}\right) = \frac{\overbrace{2310519339 \ldots 5639754240}^{113\,\,{\rm{digits}}}}{\underbrace{2288969863 \ldots 1341570561}_{113\,\,{\rm{digits}}}} \approx 1.00941448647564092749
$$
Substituting this value into equation \eqref{TTRA}, we can find a significantly better approximation of $\pi$.

The following is the Mathematica code:
\small
\begin{verbatim}
Print["Equation (34) at k = 7: ", 
  MantissaExponent[N[\[Pi]-4*(64/81),20]][[2]] // Abs, 
    " digits of \[Pi]"];
	
Print["Equation (32) at k = 7: ", 
  MantissaExponent[
    N[\[Pi]-4*(64/81+1/2*(1-1.00941448647564092749)), 
      20]][[2]]//Abs," digits of \[Pi]"];
\end{verbatim}
\normalsize
produces the following output:
\small
\begin{verbatim}
Equation (34) at k = 7: 1 digits of π
Equation (32) at k = 7: 4 digits of π
\end{verbatim}
\normalsize

Initial sequence $\left(s\right)_1^7$ helped us to find the value $\alpha_{7} = 81$. Now due to higher accuracy of equation \eqref{TTRA} we can generate the sequence in which its upper index is doubled
$$
\left(s\right)_1^{14} = \left(1,0,1,0,0,0,1,0,1,1,1,1,1,0\right)
$$
and with help of this sequence we can find the corresponding value $\alpha_{14} = 10430$.

Unfortunately, doubling the upper index $k$ does not always work. For example, if we attempt to double the upper index by using the initial sequence $\left(s\right)_1^8$, then we get $\alpha_{16} = 41722$ instead of correct value $\alpha_{16} = 41721$. Therefore, the upper index of the sequence should be slightly less than two.

The two-terms approximation \eqref{TTRA} doubles the number of digits of $\pi$ as compared to the single-term approximation \eqref{STRA}. This means that using the sequence $\left(s\right)_1^k$ we can obtain all sequences $\left(s\right)_1^{k + 1}$, $\left(s\right)_1^{k + 2}$, $\left(s\right)_1^{k + 3}$, etc. up to $\left(s\right)_1^{k_0}$, where $k_0$ is an integer slightly smaller than $2k$. Our numerical results show that doubling value $k$ does not always provide the correct sequence as a few binary digits at the end of the sequence $\left(s\right)_1^{2k}$ occasionally may not be correct. However, when we use the empirical equation
\[
k_0 = \left\lfloor\left(2 - \frac{1}{32}\right)k\right\rfloor,
\]
then the corresponding sequence $\left(s\right)_1^{k_0}$ is a sub-sequence of the infinite sequence \eqref{S4BO} and, therefore, it is appeared to be correct. It is interesting to note that the number $32$ in this equation is the largest integer that we found on the basis of our numerical results.

The following is a Mathematica code that shows number of digits of $\pi$ at given iteration number $n$ and integer $k$:
\small
\begin{verbatim}
Clear[str,sps,k,\[Gamma],\[Alpha],lst,\[Eta]]
 
(* String for conversion of 1/\[Pi] to sequence *)
str="ImportString[ToString[BaseForm[N[1/piAppr,k0],2]],
  \"Table\"][[1]][[1]]";

(* String for space separation *)
sps[n_]:=Module[{m=1,sps=" "}, 
  While[m<n,sps=StringJoin[sps," "];m++];If[m==n,sps]];

(* Converting number to string with length q *)
cnv2str[p_,q_]:=Module[{},StringTake[StringJoin[ToString[p],sps[q]],q]]

(* Defining \[Eta]-function *)
\[Eta][n_,x_,k_]:=Module[{K=k/1.5,y=x},y=N[(2*y)/(1-y^2),K];cntr = 1; 
  While[cntr<n,y=(2*y)/(1-y^2);cntr++];y];

(* Define \[Alpha]_1, \[Alpha]_2 and \[Alpha]_3] *)
\[Alpha][1]=1;
\[Alpha][2]=2;
\[Alpha][3]=5;

(* Input values *)
k=3;\[Gamma]=\[Alpha][3];

(* Heading *)
Print["-------------------------------"];
Print["Iteration | k", sps[5], "| Digits of \[Pi]"];
Print["-------------------------------"];

n=1;
While[n<=12,

  intR=1/\[Gamma];
  k0=\[LeftFloor](2-1/32)*k\[RightFloor];

  piAppr=4*(2^(k-1)*intR+1/2*(1-\[Eta][k-1,intR,k0]));

  (* Extracting the sequence {1,0,1,0,0,0,1...} *)
  lst=RealDigits[ToExpression[str]][[1]][[1;;k0]];
  
  (* Main computation *)
  K=k+1;
  While[K<=k0,\[Gamma]=
    2*\[Gamma]+lst[[K]];\[Alpha][K]=\[Gamma];K++];k=k0;

  (* Aligned output" *)
  Print[cnv2str[n,5],sps[4]," | ",cnv2str[k,5]," | ",
    MantissaExponent[N[\[Pi]-piAppr,k0]][[2]]//Abs];

  n++];
\end{verbatim}
\normalsize
This code generates the output:
\small
\begin{verbatim}
------------------------------
Iteration | k    | Digits of π
------------------------------
1         | 5    | 1
2         | 9    | 2
3         | 17   | 4
4         | 33   | 9
5         | 64   | 20
6         | 126  | 38
7         | 248  | 75
8         | 488  | 149
9         | 960  | 293
10        | 1890 | 577
11        | 3720 | 1137
12        | 7323 | 2240
\end{verbatim}
\normalsize

As we can see from the third column, the number of digits of $\pi$ doubles at each iteration.

It should be noted that the tangent function $\tan\left(2^{k - 1}/\alpha_k\right)$ can also be computed by combining together the equations \eqref{L4TF} and \eqref{EF}. In particular, taking $\sigma \le k - 1$ we can apply the identity
\begin{equation}\label{EWS} % equation with sigma
\tan\left(\frac{2^{k - 1}}{\alpha_k}\right) = \tan\left(2^\sigma\,\frac{2^{k - 1 - \sigma}}{\alpha_k}\right) = \eta_\sigma\left(\tan\left(\frac{2^{k - 1 - \sigma}}{\alpha_k}\right)\right).
\end{equation}

The fol1lowing is the Mathematica code where the identity \eqref{EWS} is implemented by using the equations \eqref{L4TF} and \eqref{EF} at $\sigma = 100$ and $k=7323$:
\small
\begin{verbatim}
Clear[tanF]

(* Computing tangent *)
tanF[n_,x_,k_]:=Module[{p=0,q=0,m=1,x0=N[x,k]},xSq=x0^2;
  While[m<=n,
    r=(-1)^(m-1)/(2*m-1)!*x0;
    p=p+r;
    q=q+2^(2*m-1)*r;
    m++;x0=x0*xSq];
    (2*p^2)/q];

(* Case k = k0 = 7323 *)
n=1;
\[Sigma]=100;
While[n<=10,numb=tanF[n,2^(k0-1-\[Sigma])/\[Alpha][k0],k0]; 
  Print["n = ",n,": ",MantissaExponent[\[Pi]-
  4*(2^(k0-1)/\[Alpha][k0]+ 
  1/2*(1-\[Eta][\[Sigma],numb,k0]))][[2]]//Abs,
  " digits of \[Pi]"];n++];
\end{verbatim}
\normalsize
This code returns the following output:
\small
\begin{verbatim}
n = 1: 60 digits of π
n = 2: 121 digits of π
n = 3: 182 digits of π
n = 4: 244 digits of π
n = 5: 306 digits of π
n = 6: 368 digits of π
n = 7: 430 digits of π
n = 8: 492 digits of π
n = 9: 554 digits of π
n = 10: 617 digits of π
\end{verbatim}
\normalsize
This output shows high convergence rate. Specifically, we can see that each increment $n$ in equation \eqref{L4TF} contributes for more than $60$ digits of $\pi$. Generally, this convergence rate has no upper limitation. However, in order to increase $\sigma$ in equation \eqref{EWS}, we have to increase the value of the integer $k$.

Although this convergence rate is significantly higher than those in some modern algorithms for computing digits of $\pi$ \cite{Agarwal2013}, even without nested square roots of 2 the proposed method at large values $k$ still requires a powerful computer to compute the coefficients $\alpha_k$ by using the binary form of the $1/\pi$.

The presence of the multiplier $2^{k - 1 - \sigma}$ in equation \eqref{EWS} is continuously divisible by $2$. Therefore, its application may be advantageous for more rapid computation since each multiplication by $2$ implies just a binary shift of the mantissa.

The Mathematica code below shows the values of $\alpha_k$ at given $k$:
\small
\begin{verbatim}
(*Heading*)
Print["-----------------"];
Print["k",sps[5],"| \[Alpha][k]"];
Print["-----------------"];

k=2;
While[k<=25,

  (*Aligned output" *)
  Print[cnv2str[k,5], " | ", \[Alpha][k]];
  k++];
\end{verbatim}
\normalsize
This code returns the following output:
\small
\begin{alltt}
-----------------
k     | α\(\sb{k}\)
-----------------
2     | 2
3     | 5
4     | 10
5     | 20
6     | 40
7     | 81
8     | 162
9     | 325
10    | 651
11    | 1303
12    | 2607
13    | 5215
14    | 10430
15    | 20860
16    | 41721
17    | 83443
18    | 166886
19    | 333772
20    | 667544
21    | 1335088
22    | 2670176
23    | 5340353
24    | 10680707
25    | 21361414
\end{alltt}
\normalsize
As we can see, all numbers $\alpha_{k}$ are the same as those reported in \cite{WolframCloud}.

Since the constants $\alpha_k$ have been computed already, we can use them to validate the formula \eqref{BFF4P} for $1/\pi$ in the binary form. The Mathematica code below:
\small
\begin{verbatim}
f[K_]:=N[Sum[(1/10^(k+1))*Mod[\[Alpha][k],2],{k,1,K}],K]; 

Print["-------------------------------------------------------------"]
Print["k",sps[2]," | Binary output"];
Print["-------------------------------------------------------------"]

k := 1; 
While[k<=5,Print[10*k,sps[2],"| ",Subscript[f[10*k],2]];k++]

Print["-------------------------------------------------------------"]

Print["Built-in Mathematica:"]
Print["1/\[Pi]=",
  Subscript[StringJoin["[",
    StringSplit[ToString[BaseForm[N[1/Pi,15],2]]][[1]], "...]"],
	    2]];
\end{verbatim}
\normalsize
returns the output:
\small
\begin{alltt}
-------------------------------------------------------------
k  | Binary output
-------------------------------------------------------------
10 | 0.01010001011\(\sb{2}\)
20 | 0.010100010111110011000\(\sb{2}\)
30 | 0.0101000101111100110000011011011\(\sb{2}\)
40 | 0.01010001011111001100000110110111001001110\(\sb{2}\)
50 | 0.010100010111110011000001101101110010011100100010000\(\sb{2}\)
-------------------------------------------------------------
Built-in Mathematica:
\(1/\pi\) = [0.010100010111110011000001101101110010011100100010000...]\(\sb{2}\)
\end{alltt}
\normalsize
according to the equation \eqref{BFF4P}. The original binary representation of the number $1/\pi$ generated by built-in Mathematica is also shown for comparison (see also \cite{OEIS2007} for binary sequence for $1/\pi$).

\section{Conclusion}

We consider the properties of the two-term Machin-like formula for $\pi$ and propose its two-term rational approximation \eqref{TTRA}. Using this approach, we developed an efficient algorithm for computing digits of $\pi$ with squared convergence. The constants $\alpha_{k}$ in this approximation are computed without nested square roots of $2$.

\section*{Acknowledgment}

This work was supported by National Research Council Canada, Thoth Technology Inc., York University and Epic College of Technology.

\end{document}